\documentclass[sn-mathphys]{sn-jnl}

\jyear{2021}%

\usepackage{dsfont}
\usepackage{empheq}
\usepackage{lmodern}
\usepackage{amsmath}
\usepackage{bbm} 
\usepackage{comment}

\newcommand{\WD}{\textup{(WD)}} 
\newcommand{\SD}{\textup{(SD)}} 

\newtheorem{theorem}{Theorem}[section]
\newtheorem{corollary}{Corollary}

\newtheorem{lemma}[theorem]{Lemma}

\newtheorem{definition}[theorem]{Definition}
\newtheorem{remark}{Remark}

\begin{document}

\title[Finite-time stabilization for degenerate singular parabolic equations]{Finite-time stabilization via impulse control of degenerate singular parabolic equations}


\author*[1]{\fnm{Walid} \sur{Zouhair}}\email{walid.zouhair.fssm@gmail.com}

\author[2]{\fnm{Ghita} \sur{El Guermai}}\email{elguermai.ghita@ucd.ac.ma}

\author[3]{\fnm{Ilham} \sur{Ouelddris}}\email{ouelddris.ilham@gmail.com}

\affil*[1]{\orgdiv{Department of Mathematics}, \orgname{Ibn Zohr University, Faculty of Applied Sciences Ait Melloul},
\orgaddress{\street{Route Nationale N10}, \city{Azrou}, \postcode{B.P. 6146}, \country{Morocco}}}

 \affil[2]{\orgname{University Chouaib Doukkali,} Faculty Polydisciplinary Sidi Benour,
\orgaddress{\city{El Jadida}, \country{Morocco}}}

\affil[3]{\orgdiv{LMDP, UMMISCO (IRD-UPMC)}, \orgname{Cadi Ayyad University, Faculty of Sciences Semlalia},
\orgaddress{\city{Marrakesh}, \postcode{B.P. 2390}, \country{Morocco}}}

\affil[]{}

\affil[]{\orgdiv{\Large Dedicated to the memory of Professor Hammadi Bouslous}}


\abstract{
This paper examines the impulse controllability of degenerate singular parabolic equations through a modern framework focused on finite-time stabilization. Furthermore, we provide an explicit estimate for the exponential decay of the solution. The proof of our main result combines a logarithmic convexity estimate with specific spectral properties. Finally, we establish the existence and uniqueness of the minimal norm impulse control associated with the system.
}



\maketitle

\section{Introduction}\label{sec1}
In this work, we investigate the degenerate singular parabolic equation within the framework of null controllability via impulse controls supported in a nonempty open subset of the physical
domain, which are well-studied in the case of nondegenerate and nonsingular heat equation. More precisely, we consider the following impulse controlled system
\begin{equation}\label{1.1}
\partial_t y - (x^{\alpha} y_x)_x - \frac{\mu}{x^{\beta}}\,y = 0
\quad \text{in } (0,1)\times\bigl(0,T\bigr)\setminus\{\tau_k:k\ge0\},
\end{equation}
with initial condition $y_0 \in L^2(0,1)$, homogeneous boundary conditions at $x=1, (y(1, t)=0)$, and a left boundary at $x=0$ that may be
\emph{weakly} or \emph{strongly} degenerate:
\[
y(1,t)=0,\quad
\begin{cases}
y(0,t)=0, & \WD,\\[2pt]
(x^{\alpha} y_x)(0,t)=0, & \SD,
\end{cases}
\qquad t\in(0,T).
\]
The control acts at discrete times $\tau_k$ on an open subset $\omega\subset(0,1)$ that touches
the degeneracy point:
\begin{equation}\label{cond}
\omega=(0,a)\quad\text{for some } a\in(0,1),
\end{equation}
through the impulse update
\begin{equation*}
y(\cdot,\tau_k)=y(\cdot,\tau_k^-)+\mathbf{1}_\omega\,L_k\bigl(y(\cdot,t_k)\bigr),
\end{equation*}
where $L_k:L^2(0,1)\to L^2(\omega)$ are bounded linear operators to be designed.

Fix $T>0$ and $b>1$, and define the increasing sequence $\{t_k\}_{k\ge0}$ by
\begin{equation*}
t_k:=T\Bigl(1-\frac{1}{b^k}\Bigr),\qquad k\ge0,
\end{equation*}
so that $t_k$ converges to $T$ as $k$ goes to infinity. Each impulse time is placed at the midpoint
\begin{equation*}
\tau_k:=\frac{t_k+t_{k+1}}{2}\in(t_k,t_{k+1}).
\end{equation*}

It is worth mentioning that the non-impulsive controllability analysis of parabolic equations with the least possible controls in the context of nondegenerate problems have been extensively studied in the literature, we refer to \citep{ABD2006, ABDG2009, AJMW2024, CGKM2023} and the references therein. The degenerate/singular parabolic problems with distributed controls has been also analyzed in several recent papers. Among them, let us mention \citep{BCG, BAJS, CMV}, focuses on the controllability problem of degenerate parabolic equations. Additionally, works such as \citep{AFS, FS, V} treated cases involving singular potentials. Similarly to distributed controllability, there are also results concerning boundary controllability for degenerate/singular parabolic problems, see for instance \cite{BHV'2022,GL'2024}. Notably, a common strategy across all referenced papers is the use of suitably adapted Carleman estimates to address the singularities within the equation and establish the controllability properties. 

Impulsive controllability is a specialized concept in control theory that addresses the ability to manage and direct the behavior of a dynamical system through control inputs applied at precise, discrete time points, commonly known as “impulse times.” Unlike continuous control approaches where inputs are constantly adjusted, impulsive controllability applies controls intermittently, focusing on specific instants or intervals. This method allows for efficient intervention in scenarios where continuous control may be impractical or unnecessary, making it highly relevant for systems requiring abrupt adjustments, such as those in robotics, aerospace, and certain automated processes. In this context, impulsive controllability was first examined for a linear heat equation with homogeneous Dirichlet and Neumann boundary conditions, as demonstrated in \cite{pkm, RBKDP2022}. This foundational work introduced a novel approach, integrating the logarithmic convexity method with the Carleman commutator approach. Subsequent research in \citep{CGMZ'22} extended these findings to encompass dynamic boundary conditions for the same system. Additionally, a numerical study in \citep{CGMZ'221} presents a constructive algorithm designed to compute the minimal \( L^2 \)-norm impulse control, providing practical insights and simulations for control implementation. 

It is worth highlighting that the controllability results established for all the aforementioned parabolic systems via impulsive control focus primarily on achieving approximate null controllability. Recently, in \cite{CGMZ2023'}, the authors provided an explicit estimate for the exponential decay of the solution under impulsive controls, directly attaining impulse null controllability a more advanced and desirable result, as this has already been established for both distributed and boundary controls.

In the current work, we extend the results from \citep{BP'18}, which dealt with a one-dimensional degenerate elliptic operator with the zero Dirichlet condition, to a more generalized case involving a degenerate parabolic operator with a singular potential given by
    \begin{equation*}
        A u := (x^{\alpha} u_x)_x + \dfrac{\mu}{x^{\beta}} u,
    \end{equation*}
under the following assumptions on $\alpha, \beta$, and $\mu$  
    \begin{itemize}
		\item  sub-critical potentials:
		\begin{equation}\label{assump1}
			\begin{aligned}
				&\alpha \in \left[ 0 , 2 \right[ , \; \; 0< \beta< 2- \alpha \; \; \mathit{and} \;  \; \mu \in \mathbb{R},\\
				&\alpha \in \left[ 0 , 2 \right[ \backslash\lbrace{1}\rbrace, \; \; \beta= 2- \alpha \; \; \mathit{and} 		\; \; \mu < \mu( \alpha);
			\end{aligned}
		\end{equation}
		\item critical potentials:
		\begin{equation}
			\alpha \in \left[ 0 , 2 \right[ \backslash\lbrace{1}\rbrace, \; \; \beta= 2- \alpha \; \; \mathit{and} 		\; \; \mu = \mu( \alpha), \label{assump2}
		\end{equation}
	\end{itemize}
	where $\mu(\alpha)= \displaystyle{\frac{(1-\alpha)^{2}}{4}}$ is the constant in the following Hardy-Poincar\'e inequality
	\begin{equation}\label{hardy1}
		\displaystyle{\int_{0}^{1}} x^{\alpha} u_{x}^{2} \,\mathrm{d}x \geq \mu(\alpha) \displaystyle{\int_{0}^{1}} \displaystyle{\frac{u^{2}}{x^{2-\alpha}}} \,\mathrm{d}x, \qquad \text{for all}\;\; u \in \mathcal{C}^{\infty}_{c}\left( 0 , 1 \right).  
	\end{equation}
In a related context, Vancostenoble established in \citep{V} an improved Hardy-Poincar\'e inequality in the sub-critical case for any $\alpha \in \left[0,2\right),$ $\nu>0$ and $0<\gamma<2-\alpha$. This inequality guarantees the existence of a positive constant $\delta= \delta(\alpha, \nu,\gamma)$ such that the following estimate holds for every $u \in \mathcal{C}_0^{\infty}(0,1)$
    \begin{equation}\label{hardy2}
        \int_0^1 x^{\alpha} u_x^2 \,\mathrm{d}x + \delta \int_0^1 u^2 \,\mathrm{d}x \geq \int_0^1 \dfrac{\mu(\alpha)}{x^{2-\alpha}} u^2 \,\mathrm{d}x + \nu \int_0^1 \dfrac{u^2}{x^{\gamma}} \,\mathrm{d}x.
    \end{equation}
    
To obtain the impulse null controllability result, we adopt the same strategy as in \cite{Coron2017, pkm1}. Specifically, we establish the finite-time stabilization of the impulsive system \eqref{1.1} and subsequently conclude the controllability result. To fulfill this objective, we recall the following definition.
    \begin{definition}
        The system \eqref{1.1} is said to be finite-time stabilizable in time $T>0$, if there exist control operators $\mathcal{L}_{k}$ such that for every $y_0 \in L^2(0,1)$ the corresponding solution satisfies $\lim\limits_{t\to T^-} y(t)=0$.
    \end{definition}
In Section \ref{sec4}, we provide the proof of the main result presented in this work, which is given as follows
\begin{theorem}\label{thm1.4}
Consider $\tau_{k} = \frac{t_{k} + t_{k+1} }{2},$ with $t_{k}=T\left( 1-\dfrac{1}{b^k} \right)$ and $b>1$. Assume that the control subset \(\omega\) satisfies \eqref{cond}. Assume that \eqref{assump1} or \eqref{assump2} is true. Then, the system \eqref{1.1} is finite-time stabilizable. Moreover, there exist positive constants $C,\mathcal{M}$ such that for any initial condition $y_0 \in L^2(0,1)$, the solution $y$ of the system \eqref{1.1} satisfies
\begin{equation*}
\|y(t)\| \leq C \mathrm{e}^{-\frac{1}{\mathcal{M}}\left(\frac{T}{T-t}\right)} \left\|y_{0}\right\| \qquad \text { for any } \;\;0 \leq t<T^{-}.
\end{equation*}
Furthermore,
$\displaystyle \lim _{k \rightarrow \infty}\left\|\mathcal{L}_{k}\left(y\left(\tau_{k}\right)\right)\right\|_{L^2(\omega)}=0.$
\end{theorem}
The main key to prove the above theorem is the following result of observability estimate at a specific time point  that has been recently established in \citep{MMOS}, where the method of proof combines both the logarithmic convexity and the Carleman commutator.
\begin{theorem}\label{theo1.1}
Assume that the observation region $\omega$ fulfills the geometric condition \eqref{cond} and suppose that one of the assumptions \eqref{assump1}-\eqref{assump2} holds. Then, there exist $\mathcal{C}_{1} >0$ and $\rho \in (0,1)$ such that the following estimate holds 
    \begin{equation}
        \|u(T)\| \leq\left( \mathrm{e}^{\mathcal{C}_1\left(1+ \delta + \delta T +\frac{1}{T} \right)}\|u( T)\|_{L^2(\omega)}\right)^\rho\left\|u_0\right\|^{1-\rho}.
    \end{equation}
Here, $\delta$ is the constant in the improved Hardy-Poincar\'e inequality \eqref{hardy2}, and for the constant $\mathcal{C}_1$ there is a specific expression that is provided in \citep{MMOS}. Furthermore, $u$ in this context stands for the solution of the following non-impulsive system
    \begin{equation}\label{1.2}
    \begin{cases}
        \partial_t u-\left(x^\alpha u_x\right)_x-\dfrac{\mu}{x^\beta} u=0, & \text { in }(0,1) \times(0, T) , \\ 
        u(1, t)=0,  & \text { on }(0, T)\\
        \left\{\begin{array}{ll}
        u(0, t)=0, & (W D), \\
        \left(x^\alpha u_x\right)(0, t)=0, & (S D),
        \end{array} \right. &\text { on }(0, T),\\
        u(x, 0)=u_0(x), & \text { on }(0,1).
    \end{cases}
    \end{equation}
\end{theorem}
\begin{remark}
    It is worth mentioning that the geometric condition on the control region \( \omega \) stated in \eqref{cond} is required to achieve the above result. Removing this assumption and considering any non-empty open subset \( \omega \) of \( (0,1) \) remains an open question, as the main difficulty lies in selecting an appropriate weight function to establish the observability inequality.
\end{remark}
The rest of the paper is structured as follows: Section \ref{sec2} introduces the functional setting and establishes the well-posedness of the singular degenerate equation. In Section \ref{sec4}, we derive finite-time stabilization through impulse controls, which consequently leads to the impulse null controllability of system \eqref{1.1}.
\section{Functional setting}\label{sec2}
 In this section, we introduce the functional framework associated with singular degenerate operators, aiming to clarify the problem. We present the relevant Sobolev spaces and their properties, distinguishing between the sub-critical and critical cases.

To clarify the mathematical framework, it is important to understand the role of the weighted Sobolev spaces account for the degeneracy at $x=0$, where the weight $x^\alpha$ is introduced to compensate for the loss of uniform ellipticity and ensures finite energy. 

More specifically, we begin with the following definition in the sub-critical case, where the functional spaces are explicitly adapted to the order of degeneracy.

\begin{definition}
Let the basic weighted Hilbert space $H_{\alpha}^{1}(0,1)$ be given as follows 
$$H_{\alpha}^{1}(0,1):=\{u\in L^{2}(0,1)\cap H_{loc}^{1}((0,1]) \mid x^{\alpha/2}u_{x}\in L^{2}(0,1)\},$$
associated to the norm
    \begin{equation*}
        \|u\|^{2}_{H^{1}_{\alpha}(0,1)}:= \|u\|^{2}_{L^{2}(0,1)} + \|x^{\alpha/2} u_x\|^{2}_{L^{2}(0,1)}.
    \end{equation*}
\end{definition}

 Depending on the degree of degeneracy, we adapt the boundary conditions. For weak degeneracy ($0 \le \alpha < 1$), the trace is well-defined and we impose $u(0)=0$, leading to
$$H_{\alpha,0}^{1}(0,1):=\{u\in H_{\alpha}^{1}(0,1) \mid u(0)=u(1)=0\}.$$

For strong degeneracy ($1 \le \alpha < 2$), the trace at \(x=0\) is no longer defined, instead, we impose a Neumann-type condition $(x^\alpha u_x)(0)=0$, and retain only $u(1)=0$:
$$H_{\alpha,0}^{1}(0,1):=\{u\in H_{\alpha}^{1}(0,1) \mid u(1)=0\}.$$

In this case, one can consider the linear operator
    \begin{equation*}
        Au= (x^{\alpha} u_x)_x + \dfrac{\mu}{x^{\beta}} u,
    \end{equation*}
with its domain depending on the value of $\alpha$
    \begin{equation}\label{DA1}
        \begin{aligned}
            &\text{If}\;\; 0\leq \alpha <1\\
			&D(A):=\big\{u\in H_{\alpha,0}^1(0,1)\cap H_{loc}^2((0,1])\, \mid  Au \in L^2(0,1) \big\}\\
			&\text{and if}\;\; 1< \alpha<2\\
			&D(A):=\big\{u\in H_{\alpha,0}^1(0,1)\cap H_{loc}^2((0,1])\, \mid  Au \in L^2(0,1)\quad \text{and}\quad (x^{\alpha}u_x)(0)=0 \big\}.
		\end{aligned}
  \end{equation}
Otherwise, in the critical case, it is crucial to adjust the structure of the Sobolev spaces accordingly, as outlined in the following definition.
    \begin{definition}
		Take $\alpha \in [0, 1)$, we define
            \begin{equation*}
                H_{\alpha, 0}^{1,\mu(\alpha)}(0,1):=\left\{u \in H_{\alpha}^{1,\mu(\alpha)}(0,1) \mid u(0)=u(1)=0\right\},
            \end{equation*}
		and for $\alpha \in [1, 2)$, we consider
            \begin{equation*}
                H_{\alpha, 0}^{1,\mu(\alpha)}(0,1):=\left\{u \in H_{\alpha}^{1,\mu(\alpha)}(0,1) \mid u(1)=0\right\}.
            \end{equation*}
		In this setting, $H_{\alpha}^{1}(0,1)$ is the Hilbert space defined as
            \begin{equation*}
                H_{\alpha}^{1,\mu(\alpha)}(0,1):=\left\{u \in L^{2}(0,1) \cap H_{l o c}^{1}((0,1]) \mid \int_0^1 \left( x^{\alpha} u_x^2 - \dfrac{\mu(\alpha)}{x^{2-\alpha}} u^2 \right) \,\mathrm{d}x < \infty\right\},
            \end{equation*}
		equipped to the following scalar product
            \begin{equation*}
                \left\langle u , v \right\rangle_{H_{\alpha}^{1,\mu(\alpha)}}= \int_0^1 \left( uv + x^{\alpha} u_x v_x + \dfrac{\mu(\alpha)}{x^{2-\alpha}} u v \right)\,\mathrm{d}x.
            \end{equation*}
	\end{definition}
In this context, we introduce the operator
    \begin{equation*}
        Au= (x^{\alpha} u_x)_x + \dfrac{\mu(\alpha)}{x^{2-\alpha}},
    \end{equation*}
with domain 
     \begin{equation}\label{DA2}
        \begin{aligned}
            &\text{If}\;\; 0\leq \alpha <1\\
			&D(A):=\big\{u\in H_{\alpha, 0}^{1,\mu(\alpha)}(0,1)\cap H_{loc}^2((0,1])\, \mid  Au \in L^2(0,1) \big\}\\
			&\text{and if}\;\; 1< \alpha<2\\
			&D(A):=\big\{u\in H_{\alpha, 0}^{1,\mu(\alpha)}(0,1)\cap H_{loc}^2((0,1])\, \mid  Au \in L^2(0,1)\quad \text{and}\quad (x^{\alpha}u_x)(0)=0 \big\}.
		\end{aligned}
  \end{equation}
Indeed, Vancostenoble in \citep{V} proved that the operator $(A-\delta I)$ is self-adjoint and negative in both cases. Consequently, this allows us to deduce the following regularity result
    \begin{theorem}\label{thm2.3}
        Assume that one of the assumptions \eqref{assump1} or \eqref{assump2} is valid. Then, for each initial state $u_0 \in L^2(0,1)$ there exists a unique weak solution of the system \eqref{1.2} such that
            \begin{equation*}
			u \in \mathcal{C}\left( \left[0,T \right], L^2(0,1) \right) \cap \mathcal{C}\left( \left(0,T \right], D(A) \right) \cap \mathcal{C}^1\left( \left(0,T \right], L^2(0,1) \right).
		\end{equation*}
        Moreover, if $u_0 \in D(A)$ one has
        \begin{equation*}
			u \in \mathcal{C}\left( \left[0,T \right], D(A) \right) \cap \mathcal{C}^1\left( \left[0,T \right], L^2(0,1) \right).
		\end{equation*}
    \end{theorem}
On the other hand, we rewrite system \eqref{1.1} as the impulsive Cauchy problem
\begin{equation}
\text{(ACP)} \;\; \begin{cases}
\hspace{-0.1cm} \partial_t y(t)=A y(t), \quad (0,T)\setminus \displaystyle \bigcup _{k\geq 0 }\{\tau_{k}\}, \nonumber\\
\hspace{-0.1cm} y\left(\cdot,\tau_{k}\right) = y\left(\cdot,\tau_{k}^{-}\right) + \mathds{1}_{\omega} h(t_{k}),\nonumber\\
\hspace{-0.1cm} y(0)=y_0, \nonumber
\end{cases}
\end{equation}
For all $y_0\in L^2(0,1)$, the system (ACP) has a unique mild solution given by 
    \begin{equation}\label{sol1}
        y(t) = \mathrm{e}^{t A} y_{0} + \sum_{k\geq 1} \mathds{1}_{\{t\geq \tau_{k} \}}(t)\, \mathrm{e}^{(t-\tau_{k})A} \mathds{1}_{\omega} h(t_{k}), \qquad t\in (0,T).
    \end{equation}
For further details on the existence and uniqueness of solutions for impulsive systems, see \citep{FO, SLAPWZ2}.
\subsection{Impulsive approximate controllability}
For a given $T>0$, let us take into account the system with a single pulse $\tau \in (0,T)$ given as follows
    \begin{equation}\label{1.3}
        \begin{cases}\partial_t y-\left(x^\alpha y_x\right)_x-\dfrac{\mu}{x^\beta} y=0, & \text { in }(0,1) \times(0, T) \setminus \{\tau\}, \\ 
        y(\cdot, \tau)=y\left(\cdot, \tau^{-}\right)+\mathds{1}_{\omega} h(.,\tau), &\text { in } (0,1),\\
        y(1, t)=0,  & \text { on }(0, T)\\
        \left\{\begin{array}{ll}
        y(0, t)=0, & (W D), \\
        \left(x^\alpha y_x\right)(0, t)=0, & (S D),
        \end{array} \right. &\text { on }(0, T),\\
        y(x, 0)=y_0(x), & \text { on }(0,1),
        \end{cases}
    \end{equation}
In particular, \cite{MMOS} focuses on analyzing whether it is possible to find a control function acting on \( \omega \times \{\tau\} \) that drives the solution of \eqref{1.3} from any initial state \( y_0 \) to a neighborhood of zero at the final time \( T \), a concept known as null approximate impulse controllability with a single pulse control. To clarify the problem under consideration, we recall the following definition.
\begin{definition}
The degenerate/singular system \eqref{1.3} achieves null approximate impulse controllability at the final time $T$. In other words, for any $\varepsilon > 0$ and any $y_0 \in L^2(0,1)$, a control function $h(\cdot, \tau) \in L^2(\omega)$ exists such that the solution of \eqref{1.3} satisfies
    \begin{equation*}
        \|y(\cdot, T)\| \leq \varepsilon\left\Vert y_0\right\Vert .
    \end{equation*}
\end{definition}
\noindent This implies that for every $\varepsilon >0$ and $y_0 \in L^2(0,1)$, the set defined as
\begin{equation*}
\mathcal{R}_{T, y_0, \varepsilon} :=\left\{h(\cdot, \tau) \in L^{2}(\omega): \text { the solution of }\eqref{1.3}\text { verifies }\left\Vert y(\cdot, T)\right\Vert \leq \varepsilon\left\Vert y_0\right\Vert\right\},
\end{equation*}
is nonempty. This brings us to the following result concerning the approximate impulse controllability for system \eqref{1.3}, with the proof provided in Theorem 5 in \citep{MMOS}.
\begin{corollary}\label{cor1}
Assume that \eqref{assump1} or \eqref{assump2} holds true and  the control subset $\omega$ satisfies \eqref{cond}. Then, the system \eqref{1.3} is null approximate impulse controllable at any time $T > 0$. Moreover, for any $\varepsilon > 0,$ there exists a positive constant $M$ such that
    \begin{equation}
        \dfrac{1}{M^2}\|h(\cdot,\tau)\|_{L^{2}(\omega)}^{2}+\dfrac{1}{\varepsilon^2}\|y(\cdot, T)\|^{2} \leq\left\|y_0\right\|^{2},
    \end{equation}
with an explicit expression of the constant, as provided by 
    \begin{equation*}
        M(T,\alpha,\delta,\varepsilon,\rho) :=\dfrac{1}{\varepsilon^{\frac{1-\rho}{\rho}}} \mathrm{e}^{\mathcal{C}_1\left(1+ \delta+\delta(T+\tau)+ \frac{1}{T-\tau}\right)}, \qquad \text{where}\;\; \rho \in (0,1).
    \end{equation*}
\end{corollary}
Based on this, we can derive a particular estimate for the cost of the control function that is given as follows
\begin{lemma}\label{lem2.5}
    Suppose that one of the assumptions \eqref{assump1}-\eqref{assump2} is satisfied, in addition to the condition \eqref{cond}. Then, for any $\varepsilon >0$, the cost of the null approximate impulse control function at time $T$ satisfies
    \begin{equation}
        \mathcal{R}_{T,y_0\varepsilon} \leq \mathrm{e}^{ \mathcal{C}_1\left(1+\delta+ \delta(T+\tau) +\frac{1}{T-\tau}\right)} \mathrm{e}^{\frac{C_1}{\sqrt{T-\tau}}\sqrt{\ln\left(\mathrm{e}+\frac{1}{\varepsilon^2}\right)}}.
    \end{equation}
\end{lemma}
\begin{proof}
    From Corollary \ref{cor1}, we infer that the impulse control function satisfies the following estimate
    \begin{equation*}
        \|h(\cdot,\tau)\|^2 \leq \dfrac{1}{\varepsilon^{\frac{2(1-\rho)}{\rho}}} \mathrm{e}^{2 \mathcal{C}_1\left(1+ \delta +\delta(T+\tau) +\frac{1}{T-\tau}\right)} \|y_0\|^2.
    \end{equation*}
    Denoting $\sigma =\dfrac{1-\rho}{\rho}$, it yields
    \begin{align*}
        \|h(\cdot,\tau)\|^2 
        \leq &\dfrac{1}{\varepsilon^{2\sigma }} \mathrm{e}^{2 \mathcal{C}_1 \left(1+ \delta +\delta(T+\tau) + \frac{1}{T-\tau}\right)} \|y_0\|^2\\
        \leq &\mathrm{e}^{2 \mathcal{C}_1 \left(1+ \delta +\delta(T+\tau) +\frac{1}{T-\tau} + \sigma \ln\left(\mathrm{e} + \frac{1}{\varepsilon^2}\right) \right)} \|y_0\|^2.
    \end{align*}
    Next, choose $\sigma: =\dfrac{1}{\sqrt{(T-\tau)\ln\left(\mathrm{e}+\frac{1}{\varepsilon^2}\right)}}$ to obtain that
    \begin{equation}\label{eqq1}
        \|h(\cdot,\tau)\|^2 \leq \mathrm{e}^{2 \mathcal{C}_1 \left(1+ \delta+ \delta(T+\tau) +\frac{1}{T-\tau}\right)} \mathrm{e}^{\frac{2 \mathcal{C}_1}{\sqrt{T-\tau}} \sqrt{\ln\left(\mathrm{e} + \frac{1}{\varepsilon^2}\right)}}\|y_0\|^2.
    \end{equation}
\end{proof}
It can be confirmed that the operator $-A$ is both densely defined and closed, featuring a compact resolvent. Additional information can be consulted in \citep{V}. Consequently, the following spectral decomposition holds true.
	\begin{lemma}\label{lmsp1}
		There exists a countable family of eigenfunctions $(\phi_k)_{k \geq 1}$ associated with eigenvalues $(\lambda_k)_{k \geq 1}$ forming a Hilbert basis for $L^2(0,1)$. Moreover, the corresponding eigenvalues satisfy
        \begin{equation*}
            0<\lambda_1< \lambda_2<...<\lambda_k \rightarrow \infty \qquad \text{as} \;\; k \rightarrow \infty.
        \end{equation*}
	\end{lemma}

\section{Finite time stabilization}\label{sec4}
In what follows, we denote by $(\phi_k)_{k \geq 1}$ the orthonormal eigenfunctions in $L^2(0,1)$ associated with the eigenvalues $(\lambda_k)_{k\geq1}$ as provided by Lemma \ref{lmsp1}. It can be verified that the eigenvalues satisfy $\lambda_k \sim C(\alpha) k^2$, a proof of which can be found in \cite{Gal2015}, Theorem 2.16. This establishes the existence of a positive constant $C_0$, leading to the following estimate
\begin{equation}\label{equa34}
   \mathrm{card}\left\lbrace \lambda_{j}\leq \Lambda\right\rbrace=\sum_{\lambda_{j}\leq \Lambda} 1\leq C_0 \Lambda^{\frac{1}{2}}.  
\end{equation}
For simplicity, we adopt the notation $\left\langle \cdot , \cdot \right\rangle$ for the scalar product in the Hilbert space $L^2(0,1)$ and $\|\cdot\|$ for the corresponding norm.  Now, we introduce an increasing sequence defined as
\begin{equation}\label{equ34}
    t_{k} :=T\left( 1-\dfrac{1}{b^k} \right), \qquad \text{with}\;\; b>1,
\end{equation}
such that $t_k \rightarrow T$ as $k \rightarrow \infty$. Additionally, we denote by $\mathcal{L}_{k}$ the following linear operator
\begin{align}\label{eq35}
    \mathcal{L}_{k}:\;  L^2(0,1) &\rightarrow L^2(\omega) \nonumber\\
                    v \quad&\mapsto \sum _{\lambda_{j}\leq \Lambda_{k}}\left\langle v, \phi_{j} \right\rangle  h_{j}.
\end{align}
Here, $\Lambda_{k}:=\lambda_{1}+\dfrac{\eta}{T}\dfrac{b^{2k+1}}{b-1}$ with $\eta>1$ and $h_{j}$ is the impulse control of the following degenerate singular equation associated with the eigenfunction $\phi_{j}$.

\begin{empheq}[left = \empheqlbrace]{alignat=2}
\begin{aligned}\label{equation28}
&\partial_t y_{j}-\left(x^\alpha \partial_x y_{j}\right)_x-\frac{\mu}{x^\beta} y_{j}=0, &&\text { in } (0,1) \times(t_{k}, t_{k+1}) \backslash\{\tau_{k}\},\\
&y_{j}(\cdot, \tau_{k})=y_{j}\left(\cdot, \tau_{k}^{-}\right)+\mathds{1}_{\omega} h_{j}(\cdot,t_{k}), &&\text { in }  (0,1),\\
&y_{j}(1, t)=0,  && \text { on }(0, T)\\
&\left\{\begin{array}{ll}
y_{j}(0, t)=0, & \hspace{0.5cm}(W D), \\
\left(x^\alpha \partial_x y_{j}\right)(0, t)=0, & \hspace{0.5cm}(S D),
\end{array} \right. &&\text { on }(0, T),\\
&y_j(x, t_k)=\phi_j(x), && \text { on }(0,1),
\end{aligned}
\end{empheq}
At this point, we set $\varepsilon^2=\dfrac{\mathrm{e}^{-\eta b^{k}}}{\sum\limits_{\lambda_{j}\leq \Lambda_{k}}1}$ and apply  Corollary \ref{cor1} along with Lemma \ref{lem2.5} to establish
\begin{equation}\label{equa29}
    \|y_{j}\|^2\leq \dfrac{\mathrm{e}^{-\eta b^{ k}}}{\sum\limits_{\lambda_{j}\leq \Lambda_{k}}1}.
\end{equation}
and
\begin{equation}\label{equa30}
\|h_{j}\|_{L^2(\omega)}^2\leq\mathrm{e}^{4 C_1 \left(1+ \delta + \delta( t_{k+1} + t_k)+ \frac{1}{t_{k+1}-t_{k}}\right)}\mathrm{e}^{\frac{2 \sqrt{2} C_1}{\sqrt{t_{k+1}-t_{k}}}\sqrt{\ln \left( \mathrm{e}+\mathrm{e}^{\eta b^{k}}\sum\limits_{\lambda_{j}\leq \Lambda_{k}}1 \right)}}
\end{equation}
First of all, we provide the proof for the following essential lemma 
\begin{lemma}\label{lem}
    Under the assumption that one of the conditions \eqref{assump1}-\eqref{assump2} holds and assuming that $\omega$ satisfies the condition \eqref{cond}, the function $t \mapsto e^{-\delta t} \|y(\cdot,t)\|$ is non-increasing on the interval $(t_k,t_{k+1})$.
\end{lemma}
\begin{proof}
    First, let us consider $y \in D(A)$. Then, by using integration by parts one can obtain for each $t \in (t_k,t_{k+1})$ that
    \begin{align*}
        &\partial_t\left(\mathrm{e}^{-2\delta t} \|y(\cdot,t)\|^2\right)\\
        = &-2 \mathrm{e}^{-2\delta t}\left( \delta \int_0^1 y^2(x,t) \,\mathrm{d}x - \int_0^1 \partial_t y(x,t) y(x,t) \,\mathrm{d}x\right)\\
        = &-2 \mathrm{e}^{-2\delta t}\left( \delta \int_0^1 y^2(x,t) \,\mathrm{d}x - \int_0^1 \left( (x^{\alpha} y_x(x,t))_x + \dfrac{\mu}{x^{\beta}} y(x,t) \right) y(x,t) \,\mathrm{d}x\right)\\
        = &-2 \mathrm{e}^{-2\delta t}\left( \delta \int_0^1 y^2(x,t) \,\mathrm{d}x + \int_0^1 x^{\alpha} y_x^2(x,t) \,\mathrm{d}x - \int_0^1 \dfrac{\mu}{x^{\beta}} y^2(x,t) \,\mathrm{d}x \right)\\
    \end{align*}
    By applying the improved Hardy-Poincar\'e \eqref{hardy2} in the sub-critical case with $\beta<2-\alpha$ and the classical Hardy inequality \eqref{hardy1} in the critical case when $\beta=2-\alpha$, one can observe that
    \begin{equation*}
        \partial_t\left(e^{-2\delta t} \|y(\cdot,t)\|^2\right) \leq 0.
    \end{equation*}
    The density of $D(A)$ in $L^2(0,1)$ allows us to complete the proof.
\end{proof}
\subsection{Proof of Theorem \ref{thm1.4}}
To start, let us consider a sequence of real numbers $(a_j)_{j \geq 1}$ with the property that $\displaystyle{\sum_{j \geq 1}} a_{j} \geq 1$. Our main objective is to provide an estimation for the solution $y$ of the system \eqref{1.1} on the interval $\left(t_{k} , t_{k+1} \right)$, where the initial data is defined as $\displaystyle y\left(t_{k}\right) := \sum_{j \geq 1} a_{j} \phi_{j} \in L^2(0,1)$. To achieve this, we investigate the following two systems.
\begin{empheq}[left = \empheqlbrace]{alignat=2}\label{eq29}
\begin{aligned}
&\partial_{t} v-(x^{\alpha} v_x)_x - \dfrac{\mu}{x^{\beta}}v=0, &&\text { in } (0,1) \times(t_{k}, t_{k+1}) \backslash\{\tau_{k}\},\\
&v(\cdot, \tau_{k})=v\left(\cdot, \tau_{k}^{-}\right)+\mathds{1}_{\omega} \sum_{\lambda_{j} \leq \Lambda_{k}} a_{j} h_{j}, &&\text { in } (0,1),\\
&v(1,t)=0, &&\text { on } (t_k,t_{k+1}),\\
&\left\{\begin{array}{ll}
        v(0, t)=0, &\hspace{0.5cm}(W D), \\
        \left(x^\alpha v_x\right)(0, t)=0, &\hspace{0.5cm}(S D),
        \end{array} \right. &&\text { on }(t_k, t_{k+1}),\\
& v(\cdot, t_{k})=\sum_{\lambda_{j} \leq \Lambda_{k}} a_{j} \phi_{j} && \text{ on } (0,1).
\end{aligned}
\end{empheq}
and
\begin{empheq}[left = \empheqlbrace]{alignat=2}\label{eq28}
\begin{aligned}
&\partial_{t} \varphi-(x^{\alpha} \varphi_x)_x - \dfrac{\mu}{x^{\beta}}\varphi=0, &&\text { in } (0,1) \times(t_{k}, t_{k+1}),\\
&\varphi(1,t)=0, &&\text { on } (t_k,t_{k+1}),\\
&\left\{\begin{array}{ll}
        \varphi(0, t)=0, &\hspace{0.5cm}(W D), \\
        \left(x^\alpha \varphi_x\right)(0, t)=0, &\hspace{0.5cm}(S D),
        \end{array} \right. &&\text { on }(t_k, t_{k+1}),\\
& \varphi(\cdot, t_{k})= \sum_{\lambda_{j} > \Lambda_{k}} a_{j} \Phi_{j}, && \text{ on } (0,1).
\end{aligned}
\end{empheq}
The solutions of the systems mentioned above are
\begin{equation}\label{eq31}
v(t) = \sum_{\lambda_{j} \leq \Lambda_{k}} a_{j} \,y_{j}(t),  
\end{equation}
and,
\begin{equation}\label{eq32}
\varphi(t)=\sum_{\lambda_{j}>\Lambda_{k}} a_{j} \mathrm{e}^{-\lambda_{j}\left(t-t_{k}\right)} \phi_{j}.
\end{equation}
where $y_{j}$ is the solution of the system \eqref{equation28}. Therefore, combining \eqref{equa29} and \eqref{eq31} leads to
\begin{equation}\label{eq42}
\left\|v\left(t_{k+1}\right)\right\|^{2} \leq \sum_{\lambda_{j} \leq \Lambda_{k}}\mid a_{j}\mid^{2} \frac{\mathrm{e}^{-\eta b^{ k}}}{\displaystyle\sum_{\lambda_{j} \leq \Lambda_{k}} 1} \leq \mathrm{e}^{- \eta b^{ k}}\left\|y\left(t_{k}\right)\right\|^{2},
\end{equation}
on the other hand, based on \eqref{eq32}, one can get that
\begin{equation*}
\left\|\varphi\left(t_{k+1}\right)\right\|^{2} \leq e^{-2\Lambda_{k}\left(t_{k+1}-t_{k}\right)}\left\|y\left(t_{k}\right)\right\|^{2}.
\end{equation*}
Given that $ \Lambda_{k}\left(t_{k+1}-t_{k}\right)= \lambda_{1}\left(t_{k+1}-t_{k}\right) + \eta b^{ k}$ and noting that $\lambda_{1}\left(t_{k+1}-t_{k}\right)>0$, we deduce that
\begin{equation*}
\left\|\varphi\left(t_{k+1}\right)\right\|^{2} \leq e^{-2\eta b^{ k}}\left\|y\left(t_{k}\right)\right\|^{2}.
\end{equation*}
Using the fact that $y\left(t_{k+1}\right) = v\left(t_{k+1}\right)  + \varphi\left(t_{k+1}\right)$ enables us to to write
\begin{equation*}
\left\|y\left(t_{k+1}\right)\right\|^{2} \leq 2\left(\left\|v\left(t_{k+1}\right)\right\|^{2}+\left\|\varphi\left(t_{k+1}\right)\right\|^{2}\right) \leq  2\mathrm{e}^{- \eta b^{ k}}\left(  1+\mathrm{e}^{-\eta b^k}\right)\left\|y\left(t_{k}\right)\right\|^{2}.
\end{equation*}
Since $2\left(  1+\mathrm{e}^{-\eta b^k}\right)\leq 4 \leq \mathrm{e}^2$, then we obtain
\begin{equation}
   \left\|y\left(t_{k+1}\right)\right\|^{2} \leq\mathrm{e}^{2-\eta b^k}.
\end{equation}
By induction for any $k \geq 0$,
\begin{equation}\label{equ46}
    \left\|y\left(t_{k}\right)\right\|^{2} \leq \mathrm{e}^{2k- \eta b^{ k}}\left\|y\left(t_{0}\right)\right\|^{2},
\end{equation}
Following this, we aim to provide an estimate for the control function $\mathcal{L}_{k}$ explicitly defined in \eqref{eq35}, corresponding to the solution $y$ of the system \eqref{1.1}. This can be accomplished through the application of the Cauchy-Schwarz inequality
\begin{equation}\label{eqq2}
\begin{aligned}
\left\|\mathcal{L}_{k}\left(y\left(t_{k}\right)\right)\right\|_{\omega}^{2} 
=&\left\| \sum_{\lambda_{j}\leq \Lambda_{k}}  a_{j} h_{j}\right\|_{\omega}^{2}\\
\leq &\int_{\omega}\left(\sum_{\lambda_{j} \leq \Lambda_{k}}\mid a_{j}\mid \mid h_{j} \mid\right)^{2} \,\mathrm{d}x\\
\leq &\sum_{\lambda_{j} \leq \Lambda_{k}}\mid a_{j}\mid^{2} \sum_{\lambda_{j} \leq \Lambda_{k}}\left\|h_{j}\right\|_{\omega}^{2}\\
\leq &\left\|y\left(t_{k}\right)\right\|^{2} \sum_{\lambda_{j} \leq \Lambda_{k}}\left\|h_{j}\right\|_{\omega}^{2}.
\end{aligned}
\end{equation}
Next, from the estimate \eqref{equa30} along with applying Young's inequality we infer that
\begin{align*}
    \sum_{\lambda_{j} \leq \Lambda_{k}} \left\|h_{j}\right\|_{\omega}^{2}
    \leq &\sum_{\lambda_{j} \leq \Lambda_{k}} \mathrm{e}^{4 C_1 \left(1+ \delta+\delta(t_{k+1} + t_k)+ \frac{1}{t_{k+1}-t_{k}}\right)}\mathrm{e}^{\frac{2 \sqrt{2} C_1}{\sqrt{t_{k+1}-t_{k}}}\sqrt{\ln \left( \mathrm{e}+\mathrm{e}^{\eta b^{k}}\sum\limits_{\lambda_{j}\leq \Lambda_{k}}1 \right)}}\\
    \leq &\mathrm{e}^{4 C_1 \left(1+ \delta+ \delta(t_{k+1} + t_k) + \frac{1}{t_{k+1}-t_{k}}\right)}\mathrm{e}^{\frac{4 C_1}{\sqrt{t_{k+1}-t_{k}}}\sqrt{\ln \left(\mathrm{e}^{\eta b^{k}}\sum\limits_{\lambda_{j}\leq \Lambda_{k}}1 \right)}} \sum_{\lambda_{j} \leq \Lambda_{k}} 1 \\
    \leq &\mathrm{e}^{4 C_1\left(1+ \delta+ \delta(t_{k+1} + t_k) + \frac{1}{t_{k+1}-t_{k}}\right)}\mathrm{e}^{\frac{8 C_1^2}{t_{k+1}-t_{k}}+ \dfrac{1}{2}\ln \left(\mathrm{e}^{\eta b^{k}}\sum\limits_{\lambda_{j}\leq \Lambda_{k}}1\right)} \sum_{\lambda_{j} \leq \Lambda_{k}} 1 \\
    \leq &\mathrm{e}^{4 C_1\left(1+ \delta+\delta(t_{k+1} + t_k) + \frac{1}{t_{k+1}-t_{k}}\right)+\frac{8 C_1^2}{t_{k+1}-t_{k}}} \,\mathrm{e}^{\frac{1}{2}\eta b^k} \left(\sum_{\lambda_{j} \leq \Lambda_{k}} 1\right)^{\frac{3}{2}},
\end{align*}
where we used the fact that $\displaystyle{\mathrm{e}+\mathrm{e}^{\eta b^k}\sum_{\lambda_{j} \leq \Lambda_{k}} 1\leq \left( \mathrm{e}^{\eta b^k}\sum_{\lambda_{j} \leq \Lambda_{k}}1 \right)^2}$.

Thus, by using the definition of the sequence $(t_k)_{k \geq 0}$ and the choice of $b>1$, the above estimate becomes
\begin{align}\label{eqq3}
    \sum_{\lambda_{j} \leq \Lambda_{k}} \left\|h_{j}\right\|_{\omega}^{2}
    \leq &\mathrm{e}^{\frac{1}{2}\eta b^k} \mathrm{e}^{4 C_1\left(1+\delta+ \delta T\left(2-\frac{b+1}{b^{1+k}}\right) + \frac{b^{k+1}}{T(b-1)}\right) + \frac{8 C_1^2 b^{k+1}}{T(b-1)}} \left(\sum_{\lambda_{j} \leq \Lambda_{k}} 1\right)^{\frac{3}{2}}\nonumber\\
    \leq &\mathrm{e}^{\frac{1}{2}\eta b^k} \mathrm{e}^{8 C_1\left(1+ \delta +\delta T + \frac{(1+C_1)b}{T(b-1)}\right)b^k}\left(\sum_{\lambda_{j} \leq \Lambda_{k}} 1\right)^{\frac{3}{2}}.
\end{align}
Hence, combining \eqref{equ46}, \eqref{eqq2} and \eqref{eqq3} yields
\begin{align}\label{eq48}
&\left\|\mathcal{L}_{k}\left(y\left(t_{k}\right)\right)\right\|_{\omega}^{2}\nonumber\\
\leq &C_0^{\frac{3}{2}}\mathrm{e}^{\frac{1}{2}\eta b^k} \mathrm{e}^{8 C_1\left(1+ \delta +\delta T + \frac{(1+C_1)b}{T(b-1)}\right)b^k}  \left(\lambda_1 + \dfrac{\eta}{T} \dfrac{b^{2k+1}}{b-1}\right)^{\frac{3}{4}} \left\|y\left(t_{k}\right)\right\|^{2}\nonumber\\
\leq &C_0^{\frac{3}{2}} \,\mathrm{e}^{2k- \frac{1}{2}\eta b^{k}} \mathrm{e}^{8 C_1\left(1+ \delta +\delta T + \frac{(1+C_1)b}{T(b-1)}\right)b^k} \left(\lambda_1 + \dfrac{\eta}{T} \dfrac{b^{2k+1}}{b-1}\right)^{\frac{3}{4}} \left\|y\left(t_{0}\right)\right\|^{2}.
\end{align}
Following this, we consider the following choice of $\eta >1$
\begin{equation*}
    \eta = 1+ 32 C_1 \left(1+ \delta +\delta T + \frac{(1+C_1)b}{T(b-1)}\right),
\end{equation*}
which implies that
\begin{equation*}
    -\frac{1}{2} \eta b^{ k}+8 C_1\left(1+ \delta +\delta T + \frac{(1+C_1)b}{T(b-1)}\right) b^k  \leq-\frac{1}{4} \eta b^{ k}.
\end{equation*}
This allows us to write
\begin{align*}
\left\|\mathcal{L}_{k}\left(y\left(t_{k}\right)\right)\right\|_{\omega}^{2}
\leq &C_0^{\frac{3}{2}} \,\mathrm{e}^{2k- \frac{1}{4}\eta b^{ k}}  \left(\lambda_{1}+\dfrac{\eta}{T}\dfrac{b}{b-1} \right)^{\frac{3}{4}} b^{\frac{3}{2}k} \left\|y\left(t_{0}\right)\right\|^{2}.
\end{align*}
Moreover, one has
\begin{equation*}
    b^{\frac{3}{2}k} \leq  \left( \frac{12}{\eta}\right)^{\frac{3}{2}} \mathrm{e}^{\frac{1}{8}\eta b^{k} },
\end{equation*}
one can deduce that for any $k\geq 0$,
\begin{equation}\label{ImpCont}
\left\|\mathcal{L}_{k}\left(y\left(t_{k}\right)\right)\right\|_{\omega}^{2}\leq C_{2}\mathrm{e}^{2k- \frac{1}{8}\eta b^{k}} \left\|y\left(t_{0}\right)\right\|^{2},
\end{equation}
such that $C_{2}:= \left( \left(\frac{12 C_0}{\eta}\right)^{2} \left(\lambda_{1}+\dfrac{\eta}{T}\dfrac{b}{b-1} \right)\right)^{\frac{3}{4}}.$

For all $t\geq 0$, there exist $k\geq 0$ such that $t \in [t_{k}, t_{k+1}]$. To reach our final result, we distinguish four cases in which we apply the Lemma \ref{lem}:

\noindent If $t \in [t_{0},\tau_{0})$, one has
\begin{align}
     \left\|y\left(t\right)\right\|^{2}
     \leq &\mathrm{e}^{2\delta(t-t_0)} \left\|y\left(t_{0}\right)\right\|^{2}\nonumber\\
     \leq &\mathrm{e}^{2\delta(\tau_0-t_0)} \left\|y\left(t_{0}\right)\right\|^{2}\nonumber\\
     \leq &e^{2\delta (t_1-t_0)} \left\|y\left(t_{0}\right)\right\|^{2}.
\end{align}
If $t \in [\tau_{0}, t_{1})$, then
\begin{align}
 \left\|y\left(t\right)\right\|^{2} 
 \leq &\mathrm{e}^{2\delta (t-\tau_0)} \left\|y\left(\tau_{0}\right)\right\|^{2}\nonumber\\
 \leq &2 \mathrm{e}^{2\delta (t_1-\tau_0)} \left(\left\|y\left(\tau_{0}^{-}\right)\right\|^{2} +\left\|\mathds{1}_{\omega} \mathcal{L}_{0}(y(t_{0})) \right\|^{2}\right)\nonumber\\
 \leq &2 \,\mathrm{e}^{\delta(t_1-t_0)} \left(\left\|y\left(\tau_{0}^{-}\right)\right\|^{2} + \left\| \mathcal{L}_{0}\right\|^{2} \left\| y(t_{0})\right\|^{2}\right)\nonumber\\
 \leq &2\,\mathrm{e}^{2\delta(t_1-t_0)}  \left(1 + \left\| \mathcal{L}_{0} \right\|^{2} \right) \left\|y\left(t_{0}\right)\right\|^{2};
\end{align}
If $k \geq 0$ and $t \in [t_{k},\tau_{k})$, then
\begin{align*}
     \left\|y\left(t\right)\right\|^{2} 
     \leq &\mathrm{e}^{2\delta (t-t_k)} \left\|y\left(t_{k}\right)\right\|^{2}\\
     \leq &\mathrm{e}^{\delta (t_{k+1}-t_k)} \left\|y\left(t_{k}\right)\right\|^{2}\\
     \leq &\mathrm{e}^{2\delta (t_{k+1}-t_k)}\mathrm{e}^{2k- \eta b^{ k}} \left\|y\left(t_{0}\right)\right\|^{2}.
\end{align*}
If $k \geq 0$ and $t \in [\tau_{k}, t_{k+1})$, then
\begin{align}\label{equ54}
 \left\|y\left(t\right)\right\|^{2} 
 \leq &\mathrm{e}^{2\delta (t-\tau_k)} \left\|y\left(\tau_{k}\right)\right\|^{2}\nonumber\\
 \leq &2\mathrm{e}^{\delta (t_{k+1}-t_k)} \left( \left\|y\left(\tau_{k}^{-}\right)\right\|^{2} + \left\| \mathcal{L}_{k}(y(t_{k})) \right\|^{2}_{L^2(\omega)} \right)\nonumber\\
 \leq &2\mathrm{e}^{2\delta (t_{k+1}-t_k)} \left( \left\|y\left(t_k\right)\right\|^{2} + C_2\mathrm{e}^{2k- \frac{1}{8}\eta b^{ k}} \left\|y\left(t_{0}\right)\right\|^{2}\right)\nonumber\\
 \leq &2 \left(1+C_2 \right)\mathrm{e}^{2\delta (t_{k+1}-t_k)}\mathrm{e}^{2k- \frac{1}{8}\eta b^{ k}} \left\|y\left(t_{0}\right)\right\|^{2}.
\end{align}
Based on \eqref{equ54} and selecting $b = \mathrm{e}^{ \frac{32}{\eta}}$, we derive
\begin{align*}
    \left\|y\left(t\right)\right\|^{2} &\leq 2 \left(1+C_2 \right)\mathrm{e}^{2\delta T}\mathrm{e}^{-\frac{1}{16}\eta b^{k}} \left\|y\left(t_{0}\right)\right\|^{2}\\
    &\leq 2 \left(1+C_2+\left\| \mathcal{L}_{0} \right\|^{2} \right)\mathrm{e}^{2\delta T} \mathrm{e}^{-\frac{1}{16}\eta b^{ k}} \left\|y\left(t_{0}\right)\right\|^{2}.
\end{align*}
As a result, for each $t \in \left[t_{k},t_{k+1} \right]$ with $k \geq 0$ one can get
\begin{equation}
    \left\|y\left(t\right)\right\|^{2}
     \leq \mathcal{C} \mathrm{e}^{-\frac{1}{16}\eta b^{k}} \left\|y\left(t_{0}\right)\right\|^{2}.
\end{equation}
Here, $\mathcal{C}:= 2 \left(1+C_2+\left\| \mathcal{L}_{0} \right\|^{2} \right)\mathrm{e}^{\delta T}$. In other words, one has
\begin{equation*}
    b^{k} \leq \frac{T}{T-t} \leq b^{k+1},
\end{equation*}
which gives that
\begin{equation*}
     \mathrm{e}^{-\frac{1}{16}\eta b^{ k}} \leq \mathrm{e}^{-\frac{\eta}{16 b}\frac{T}{T-t} }.
\end{equation*}
This allows us to obtain
\begin{equation*}
\|y(\cdot,t)\| \leq \mathcal{C} \mathrm{e}^{-\frac{1}{\mathcal{M}}\left(\frac{T}{T-t}\right)} \left\|y_{0}\right\| \text { for any } 0 \leq t<T,
\end{equation*}
with $\mathcal{M} := \frac{16 b}{\eta} $. 
This enables us to establish the desired inequality.
\section{Uniqueness of minimal norm impulse control and its construction}
The main purpose of this subsection is to investigate a minimal norm problem. Indeed, the system \eqref{1.1} can be formulated as follows
\begin{equation}\label{ACP1}
\begin{cases}
\hspace{-0.1cm} \partial_t y(t)-A y(t)=0, \qquad \qquad  \qquad(0,T)\setminus \displaystyle \bigcup _{k\geq 0 }\{\tau_{k}\}, \\
\hspace{-0.1cm} y(\tau_{k}) =y(\tau_{k}^-)+ \mathds{1}_{\omega} \mathcal{L}_{k}(y(t_k)),\\
\hspace{-0.1cm} y(0)=y_0. 
\end{cases}
\end{equation}
In what follows, we discuss the minimal norm impulse control problem $(\mathcal{P})$:
\begin{equation}
    N:=\inf \left\lbrace \,\|(\mathcal{L}_{k}(y(t_k)))_{k \geq 0} \|_{\ell^2}; \;(\mathcal{L}_{k}(y(t_k)))_{k\geq 0}\in \ell^2(L^2(\omega)) \;\;\text{and}\;\;y(T)=0 \,\right\rbrace,
\end{equation}
where \[\|(\mathcal{L}_{k}(y(t_k)))_{k \geq 0} \|_{\ell^2} := \left( \sum\limits_{k\geq 0}\| \mathcal{L}_{k}(y(t_k)) \|_{L^2(\omega)}^2 \right)^{\frac{1}{2}}\] 
and $y$ represents the solution of \eqref{ACP1} corresponding to $y_0$ and $\mathcal{L}_k(y(t_k))$ for each $k \geq 0$. Now, by applying Fenchel-Rockafellar theory, the dual problem of $(\mathcal{P})$ can be presented as
\[(\mathcal{Q}): \inf_{u_k \in L^2(0,1)} \mathcal{J}_k(u_k) \qquad \text{for each} \;k \geq 0.\]
For any $k \geq 0$, the functional $\mathcal{J}_k : L^2(0,1) \rightarrow \mathbb{R}$ is defined as 
        \begin{equation}
            \mathcal{J}_k(v):=\dfrac{1}{2}\|\mathds{1}_{\omega}^*e^{(t_{k+1}-\tau_k)A}v\|_{\omega}^2+\langle y_0,e^{(t_{k+1}-t_k)}v\rangle, \qquad \forall v \in L^2(0,1).
        \end{equation}
However, the functional $\mathcal{J}_k$ is not coercive as proved in the following Lemma.
\begin{lemma}\label{coerv}
    The functional $\mathcal{J}_k$ is not coercive in $L^2(0,1)$.
\end{lemma}
\begin{proof}
    Given the compactness of the injection \(D(A) \hookrightarrow L^2(0,1)\), it results from [Proposition 4.25] that the operator $A$ has a compact resolvent. Consequently, according to [Proposition 8.11] we infer that the system \eqref{ACP1} is not exactly controllable. Furthermore, let $y_1$ be a non-reachable point in $L^2(0,1)$. In particular, one has \[y(T,y_0,0) \neq y_1.\]
    Next, we introduce the following functional 
    \[\mathcal{J}_{1,k}(v)= \dfrac{1}{2}\|\mathds{1}_{\omega}^*e^{(t_{k+1}-\tau_k)A}v\|_{\omega}^2+\langle y_0,e^{(t_{k+1}-t_k)A}v\rangle - \left\langle y_1 , v \right\rangle.\]
    The idea is to assume by contradiction that $\mathcal{J}_k$ is coercive, which in turn implies that $\mathcal{J}_{1,k}$ is coercive leading to the reachability of $y_1$.
    
    Firstly, assume that the functional $\mathcal{J}_k$ is coercive. Our main is to prove that there exists a positive constant $C>0$ such that
    \begin{equation}\label{ineq1}
        \|\mathds{1}_{\omega}^{*} e^{(t_{k+1}-\tau_k) A} v \|_{\omega}^2 \geq C \|v\|^2, \qquad \forall v \in L^2(0,1).
    \end{equation}
    By contradiction, we assume that\eqref{ineq1} is not true, i.e. there exists a sequence $(v_n)_{n \geq 0}$ such that $\|v_n\|=1$ and 
    \[\|\mathds{1}_{\omega}^{*} e^{(t_{k+1}-\tau_k) A} v_n \|_{\omega}^2 \leq \dfrac{1}{n^2}.\]
    Indeed, if we have that \(\langle y_0,e^{(t_{k+1}-t_k)}v_n\rangle \geq 0\) in the definition of the functional \(\mathcal{J}_k\) then we can choose \(w_n := -\sqrt{n} v_n\) to achieve
    \[\mathcal{J}_k(w_n) \leq \dfrac{n}{2}\|\mathds{1}_{\omega}^*e^{(t_{k+1}-\tau_k)A}v_n\|_{\omega}^2 \leq \dfrac{1}{2n}.\]
    Otherwise, if \(\langle y_0,e^{(t_{k+1}-t_k)}v_n\rangle \leq 0\) then taking $w_n:= \sqrt{n} v_n$ leads to
    \[\mathcal{J}_k(w_n) \leq \dfrac{n}{2}\|\mathds{1}_{\omega}^*e^{(t_{k+1}-\tau_k)A}v_n\|_{\omega}^2 \leq \dfrac{1}{2n}.\]
    Choose $\delta_n :=\epsilon_n w_n$ with $\epsilon_n := \pm 1$. This leads us to the conclusion that
    \[\mathcal{J}_k(\delta_n) \rightarrow 0 \qquad \text{and}\;\; \|\delta_n\| \rightarrow \infty \;\;\text{as}\;\; n \to \infty,\]
    which allows us to infer that \(\mathcal{J}_k\) is not coercive.
    
    Secondly, let us prove the coercivity of the functional \(\mathcal{J}_{1,k}\). Given that \(\mathcal{J}_k\) is coercive, there exists a positive constant $\mathcal{M}_1$ such that
    \[\mathcal{J}_k(v) \geq \mathcal{M}_1 \|v\|^2, \qquad \text{for all} \;\; v\in L^2(0,1).\]
    Hence, by using Young's inequality one can write that
    \begin{align*}
        \mathcal{J}_{1,k}(v) 
        = &\mathcal{J}_k(v) - \left\langle y_1 , v \right\rangle\\
        \geq &\mathcal{M}_1 \|v\|^2 -\left\langle y_1 , v \right\rangle\\
        \geq &\dfrac{\mathcal{M}_1}{2} \|v\|^2 - \dfrac{1}{2 \mathcal{M}_1} \|y_1\|^2.
    \end{align*}
    This implies that \(\mathcal{J}_{1,k}\) is coercive. Moreover, since the functional \(\mathcal{J}_{1,k}\) is also continuous and strictly convex, then it has a unique minimum denoted by \(\tilde{v}_k \in L^2(0,1,)\). Therefore, applying the Euler-Lagrange equation related to $\tilde{v}_k$ yields
    \[\mathds{1}_{\omega}^{*} e^{2(t_{k+1}-\tau_k)A} \tilde{v}_k + e^{(t_{k+1}-t_k)A} y_0 - y_1=0.\]
    Since \(\mathds{1}_{\omega} \mathds{1}^{*}_{\omega} = \chi_{\omega}\), one has
    \[y_1= e^{(t_{k+1}-t_k)A} y_0 + e^{(t_{k+1}-\tau_k)A}\mathds{1}_{\omega} \left(\mathds{1}^{*}_{\omega} e^{(t_{k+1}-\tau_k)A} \tilde{v}_k\right).\]
    By taking \(v_k := \mathds{1}^{*}_{\omega} e^{(t_{k+1}-\tau_k)A} \tilde{v}_k\), one gets that 
    \begin{equation}\label{ineq9}
        y_1= y(t_{k+1}, y_0, v_k) \qquad \text{for each} \;\; k \geq 0,
    \end{equation}
    where $y$ is the solution of \eqref{ACP1} associated to $y_0$ and the impulse control $v_k$ on the interval $(t_{k+1}, t_k)$. Consequently, \eqref{ineq9} is a contradiction with the fact that $y_1$ is defined as a non-reachable point. This enables us to conclude that the functional \(\mathcal{J}_k\) is not coercive in \(L^2(0,1)\).
    \end{proof}
 As shown in the above Lemma, the functional $\mathcal{J}_k$ is not coercive due to the lack of exact controllability, which may lead to unbounded minimizing sequences and absence of a minimizer in $L^2(0,1)$. To address this, we introduce a penalization term, ensuring coercivity and existence of a unique minimizer while approximating the original problem. We thus propose the following function
\[\mathcal{J}_{\varepsilon,k}(v) := \dfrac{1}{2}\|\mathds{1}_{\omega}^*e^{(t_{k+1}-\tau_k)A}v\|_{\omega}^2+\langle y_0,e^{(t_{k+1}-t_k)A}v\rangle + \varepsilon \|y_0\|\|v\|,\]
where $\varepsilon$ satisfies 
\begin{equation}\label{choice}
    \|y(t_{k+1},y_0,0)\| > \varepsilon\|y_0\| \qquad \text{for each}\;\; k \geq 0.
\end{equation} 
\begin{lemma}\label{coerv2}
    The functional \(\mathcal{J}_{\varepsilon,k}\) satisfies the following properties
    \begin{itemize}
        \item[(i)] \(\displaystyle{\lim_{q \to \infty}} \displaystyle{\inf_{\|v\|=q}} \dfrac{\mathcal{J}_{\varepsilon,k}(v)}{\|v\|} \geq \varepsilon\|y_0\|\).\\
        \item[(ii)] \(\mathcal{J}_{\varepsilon,k}\) has a unique nonzero minimum in $L^2(0,1)$, denoted as \(v_{\epsilon,k}\).
    \end{itemize}
\end{lemma}
\begin{proof}
    To establish the result in $(i)$, we assume by contradiction the existence of \(0<\sigma <\epsilon\) together with a sequence \((v_n)_{n \geq 0}\) in $L^2(0,1)$ such that \(\displaystyle{\lim_{n \to \infty}} \| v_n\| = \infty\) and for any $n \in \mathbb{N}$
    \begin{equation}\label{ineq2}
        \dfrac{\mathcal{J}_{\varepsilon,k}(v_n)}{\|v_n\|} \leq (\varepsilon-\sigma)\|y_0\|.
    \end{equation}
    As $v_n \neq 0$ for all $n \in \mathbb{N}$, we can choose $w_n := \dfrac{v_n}{\|v_n\|}$. Therefore, from \eqref{ineq2} and Cauchy-Schwartz's inequality we arrive at
    \begin{align}\label{ineq3}
        &\overline{\lim} \,\dfrac{1}{2} \|\mathds{1}_{\omega}^{*} e^{(t_{k+1}-\tau_k)A} w_n \|^2_{\omega}\nonumber\\
        = &\overline{\lim} \,\dfrac{1}{2\|v_n\|^2} \|\mathds{1}_{\omega}^{*} e^{(t_{k+1}-\tau_k)A} v_n \|^2_{\omega}\nonumber\\
        = &\overline{\lim} \,\dfrac{1}{\|v_n\|} \left( \dfrac{\mathcal{J}_{\epsilon,k}(v_n)}{\|v_n\|} - \dfrac{1}{\|v_n\|} \left\langle y_0, e^{(t_{k+1}-\tau_k)A} v_n \right\rangle - \varepsilon\|y_0\| \right)\nonumber\\
        \leq &\overline{\lim} \,\dfrac{-\sigma}{\|v_n\|} -\dfrac{1}{\|v_n\|} \left\langle y_0, e^{(t_{k+1}-\tau_k)A} w_n \right\rangle\nonumber\\
        \leq &\overline{\lim} \,\dfrac{-\sigma}{\|v_n\|} - \dfrac{1}{\|v_n\|} \|y_0\| \|e^{(t_{k+1}-\tau_k)A} w_n\|.
    \end{align}
    From the definition of $w_n$, one can observe that \((e^{(t_{k+1}-t_k)A} w_n)_{n \geq 0}\) is bounded in $L^2(0,1)$ for each $k \geq 0$. This allows us to obtain that
    \begin{equation}\label{conveerg1}
        \overline{\lim} \,\dfrac{1}{2} \|\mathds{1}_{\omega}^{*} e^{(t_{k+1}-\tau_k)A} w_n \|^2_{\omega} = 0.
    \end{equation}
    On the other hand, there exists a subsequence of \((w_n)_{n \geq 0}\) denoted by itself such that
    \[w_n \to w \;\;\text{weakly in}\;\; L^2(0,1), \;\;\text{for some}\;\; w \in L^2(0,1).\]
    Considering the compactness of the semigroup \((e^{tA})_{t \geq 0}\), the aforementioned convergence implies
    \begin{equation}\label{converg2}
        e^{(t_{k+1}-\tau_k)A} w_n \to e^{(t_{k+1}-\tau_k)A} w \qquad \text{as}\;\; n \to \infty,
    \end{equation}
    and
    \begin{equation}\label{converg3}
        \mathds{1}_{\omega}^{*} e^{(t_{k+1}-\tau_k)A} w_n \to \mathds{1}_{\omega}^{*} e^{(t_{k+1}-\tau_k)A} w \qquad \text{as}\;\; n \to \infty.
    \end{equation}
    From the three obtained limits \eqref{conveerg1}, \eqref{converg2} and \eqref{converg3}, one can deduce that
    \[\| \mathds{1}_{\omega}^{*} e^{(t_{k+1}-\tau_k)A} w \|_{\omega}=0,\]
    and the unique continuation property allows us to get that $w=0$. Now, by turning back to the definition of \(\mathcal{J}_{\varepsilon,k}\) we find that
    \begin{align}\label{ineq5}
        \underline{\lim} \,\dfrac{\mathcal{J}_{\varepsilon,k}(v_n)}{\|v_n\|}
        = &\underline{\lim} \,\dfrac{1}{2\|v_n\|} \| \mathds{1}_{\omega}^{*} e^{(t_{k+1}-\tau_k)A} v_n \|_{\omega}^2 + \dfrac{1}{\|v_n\|} \left\langle y_0, e^{(t_{k+1}-\tau_k)A} v_n \right\rangle + \varepsilon\|y_0\|\nonumber\\
        \geq &\underline{\lim} \left\langle y_0, e^{(t_{k+1}-\tau_k)A} w_n \right\rangle + \varepsilon\|y_0\| = \varepsilon\|y_0\|,
    \end{align}
    which contradicts the assumption \eqref{ineq2}. This completes the proof.
    
    We are now prepared to prove the result in $(ii)$. Specifically, from $(i)$ it follows that the functional \(\mathcal{J}_{\varepsilon,k}\) is coercive. Moreover, it is continuous and convex which allows us to conclude that \(\mathcal{J}_{\varepsilon,k}\) admits a minimum, represented as $v_{\varepsilon,k}$, in $L^2(0,1)$, i.e.
    \begin{equation*}
        \mathcal{J}_{\varepsilon,k}(v_{\varepsilon,k}) = \min_{v \in L^2(0,1)} \mathcal{J}_{\varepsilon,k}(v).
    \end{equation*}
    
    To ensure the uniqueness of this minimum, we need to prove that \(\mathcal{J}_{\varepsilon,k}\) is strictly convex in $L^2(0,1)$. For this, we fix \(v_1,v_2 \in L^2(0,1)\) with \(v_1 \neq v_2\) and we consider $\lambda \in (0,1)$. In fact, one has
    \begin{align}\label{ineq6}
        &\mathcal{J}_{\epsilon,k}(\lambda v_1 + (1-\lambda) v_2)\nonumber\\
        = &\dfrac{1}{2} \| \mathds{1}_{\omega}^{*} e^{(t_{k+1}-\tau_k)A} (\lambda v_1 + (1-\lambda) v_2 \|_{\omega}^2 + \lambda \left\langle y_0, e^{(t_{k+1}-t_k)A} v_1\right\rangle\nonumber\\
        &+ (1-\lambda) \left\langle y_0, e^{(t_{k+1}-t_k)A} v_1\right\rangle + \varepsilon \|y_0\|\| \lambda v_1 + (1-\lambda) v_2 \|.
    \end{align}
    
    Firstly, if $v_1$ and $v_2$ are linearly dependent, the Cauchy-Schwartz inequality becomes strict and then it follows that
    \begin{equation}\label{ineq7}
        \| \lambda v_1 + (1-\lambda) v_2 \| \leq \lambda \|v_1\| + (1-\lambda) \|v_2\|.
    \end{equation}
    Combining \eqref{ineq6} and \eqref{ineq7}, one has
    \begin{equation*}
        \mathcal{J}_{\epsilon,k}(\lambda v_1 + (1-\lambda) v_2) \leq \lambda \mathcal{J}_{\epsilon,k}(v_1)+(1-\lambda) \mathcal{J}_{\epsilon,k}(v_2).
    \end{equation*}
    
    Secondly, if $v_1$ and $v_2$ are not linearly dependent then there exits $0<d \neq 1$ such that $v_1 = d v_2$. Considering
    \[H(\lambda)= \mathcal{J}_{\varepsilon,k}(\lambda v_2).\]
    Since \(v_2 \neq 0\), the unique continuation property guarantees that \(\| \mathds{1}_{\omega}^{*} e^{(t_{k+1}-\tau_k)A} v_2 \|_{\omega}^2 >0\) which affirms that $H$ is a quadratic function with a non-negative leading coefficient, implying that it is strictly convex. Coupled with \eqref{ineq7}, this establishes the strict convexity of \(\mathcal{J}_{\varepsilon,k}\).

    Thirdly, let us consider the case where either \(v_1=0\) or \(v_2=0\). Assume for instance that \(v_1=0\), we find that 
    \[\|\mathds{1}_{\omega}^{*} e^{(t_{k+1}-\tau_k)A} v_1 \|_{\omega}^2 =0 \;\;\text{and}\;\; \|\mathds{1}_{\omega}^{*} e^{(t_{k+1}-\tau_k)A} v_2 \|_{\omega}^2 >0.\]
    This implies that
    \begin{align}\label{ineq8}
        &\mathcal{J}_{\varepsilon,k}(\lambda v_1 + (1-\lambda) v_2)\nonumber\\
        = &\dfrac{(1-\lambda)^2}{2} \|\mathds{1}_{\omega}^{*} e^{(t_{k+1}-\tau_k)A} v_2 \|_{\omega}^2 + (1-\lambda) \left\langle y_0, e^{(t_{k+1}-t_k)A} v_2 \right\rangle + \varepsilon (1-\lambda) \|y_0\|\|v_2\|\nonumber\\
        < &(1-\lambda) \mathcal{J}_{\varepsilon,k}(v_2) + \lambda \mathcal{J}_{\varepsilon,k}(v_1).
    \end{align}
    As a result, \(\mathcal{J}_{\varepsilon,k}\) is shown to be strictly convex in $L^2(0,1)$. Now, let us proceed to establish that the minimum $v_{\varepsilon,k} \neq 0$. To do so, we assume by contradiction that $v_{\varepsilon,k}=0$, i.e.
    \[\mathcal{J}_{\varepsilon,k}(0)= \min_{v \in L^2(0,1)} \mathcal{J}_{\varepsilon,k}(v),\]
    which gives that
    \[\mathcal{J}_{\varepsilon,k}(0) \leq \mathcal{J}_{\varepsilon,k}( \lambda v), \;\;\; \lambda \in \mathbb{R}\;\; \text{and}\;\; v \in L^2(0,1).\]
    Thus
    \[\dfrac{\lambda^2}{2} \|\mathds{1}_{\omega}^{*} e^{(t_{k+1}-\tau_k)A} v\|^2_{\omega} + \lambda \left\langle y_0, e^{(t_{k+1}-t_k)A} v \right\rangle + \varepsilon \left\lvert \lambda \right\rvert \|y_0\|\|v\| \geq 0.\]
    Consequently, by passing to the limit as $\lambda \to 0^{+}$ and $\lambda \to 0^{-}$, we infer that for any $v \in L^2(0,1)$
    \[\left\lvert \left\langle y_0, e^{(t_{k+1}-t_k)A} v \right\rangle \right\rvert \leq \varepsilon \|y_0\|\|v\|,\]
    which implies that
    \begin{equation}
        \|y(t_{k+1},y_0,0)\| \leq \varepsilon \|y_0\|.
    \end{equation}
    This contradicts the choice of $\epsilon$ given in \eqref{choice}.
\end{proof}
Now, we are ready to state the principal result concerning the norm optimal null control by means of the minimum of \(\mathcal{J}_{\varepsilon,k}\).
\begin{theorem}
The following properties are true:
    \begin{itemize}
        \item [i)] The problem $(\mathcal{P})$ has a unique minimal norm control.
        \item [ii)] The minimal norm control $(\mathcal{L}_{k}^*)_{k\geq 0}$ satisfies that $\mathcal{L}_{k}^*=0$, $k\geq 0$ if and only if the solution of \eqref{ACP1} with $\mathcal{L}_{k}(y(t_k))=0$, $k\geq 0$ satisfies \[y(T)=0.\]
        \item [iii)] Let $v_{\varepsilon,k}$ be the unique minimum of \(\mathcal{J}_{\varepsilon,k}\), then the control given by
        \begin{equation}\label{normop}
            (u_{\varepsilon,k})_{k \geq 0}= \left( \mathds{1}_{\omega}^{*} e^{(t_{k+1}-\tau_k)A} v_{\varepsilon,k}\right)_{k \geq 0},
        \end{equation}
        is a null approximate control of \eqref{1.1}, i.e
        \[\|y(T,y_0,(u_{\varepsilon,k})_{k \geq 0}) \| \leq \varepsilon\|y_0\|.\]
        Moreover, there is $(u_k)_{k \geq 0} \in \ell^2(L^2(\omega))$ such that 
        \[ u_{\epsilon,k} \to u_k \qquad \text{weakly in} \;\; \ell^2(L^2(\omega)),\]
        and $u_k$ is a norm optimal null control. That is, $u_k$ is a solution of $(\mathcal{P})$.
    \end{itemize}
\end{theorem}
\begin{proof}
    \begin{enumerate}
    \item[i)] Since the system \eqref{1.1} is null impulse controllable and the control impulse satisfies \eqref{ImpCont} , the following set is non-empty:
        \begin{equation*}
            \mathcal{F}_{ad}:=\lbrace (\mathcal{L}_{k}(y(t_k))_{k\geq 0}\in l^2(L^2(\omega)): \;\;y(T)=0 \rbrace.
        \end{equation*}
That ensures  $\mathcal{F}_{ad} \neq \varnothing$.

    On the other hand, we have $\mathcal{F}_{ad}$ is weakly closed in $l^2(L^2(\omega))$,then the problem $(\mathcal{P})$ has a minimal norm control. Now, if we consider  $(\mathcal{R}_k)_{k\geq 0}$ and $(\mathcal{S}_k)_{k\geq 0}$ as two minimal norm controls to $(\mathcal{P})$. Then we can get
    \begin{equation}\label{ega1}
    0\leq \sum\limits_{k\geq 0}\| \mathcal{R}_{k} \|_{L^2(\omega)}^2=\sum\limits_{k\geq 0}\| \mathcal{S}_{k} \|_{L^2(\omega)}^2=N^2<\infty. 
    \end{equation} 
    Furthermore, it can be verified that $\left((\mathcal{R}_k+\mathcal{S}_k)/2\right)_{k\geq 0}$ also constitutes a minimal norm control for $(\mathcal{P})$. Hence, by applying the Parallelogram low one can derive
    \begin{align}\label{ega2}
    \begin{split}
    &\sum\limits_{k\geq 0}\| (\mathcal{R}_k-\mathcal{S}_k)/2 \|_{L^2(\omega)}^2\\
    = &\dfrac{1}{2}\left( \sum\limits_{k\geq 0}\| \mathcal{R}_{k} \|_{L^2(\omega)}^2+\sum\limits_{k\geq 0}\| \mathcal{S}_{k} \|_{L^2(\omega)}^2\right) - \sum\limits_{k\geq 0}\| (\mathcal{R}_k+\mathcal{S}_k)/2 \|_{L^2(\omega)}^2\\
    = &\dfrac{1}{2}\left( \sum\limits_{k\geq 0}\| \mathcal{R}_{k} \|_{L^2(\omega)}^2+\sum\limits_{k\geq 0}\| \mathcal{S}_{k} \|_{L^2(\omega)}^2\right) - N^2.
     \end{split}
     \end{align} 
    By \eqref{ega1} and \eqref{ega2}, we obtain 
    \[ \mathcal{R}_k=\mathcal{S}_k, \quad \forall k\in \mathbb{N}. \]
    \item[ii)] This follows directly from the definition of the problem \((\mathcal{P})\).
    \item[iii)] Let us fix $k\geq 0$ and consider $v_{\varepsilon,k}$ as the unique minimum of $\mathcal{J}_{\varepsilon,k}$, i.e. 
    \[\mathcal{J}_{\varepsilon,k}(v_{\varepsilon,k}) \leq \mathcal{J}_{\varepsilon,k}(v_{\varepsilon,k}+ \lambda v), \;\;\text{for all}\;\;\lambda\in \mathbb{R}\;\;\text{and}\;\;v\in L^2(0,1).\]
    This implies that
    \begin{align*}
        &\dfrac{1}{2} \| \mathds{1}_{\omega}^{*} e^{(t_{k+1}-\tau_k)A} v_{\varepsilon,k}\|^2_{\omega} + \left\langle y_0, e^{(t_{k+1}-t_k)A} v_{\varepsilon,k} \right\rangle + \varepsilon \|y_0\|\| v_{\varepsilon,k}\|\\
        \leq &\dfrac{1}{2} \| \mathds{1}_{\omega}^{*} e^{(t_{k+1}-\tau_k)A} (v_{\varepsilon,k} + \lambda v) \|^2_{\omega} + \left\langle y_0, e^{(t_{k+1}-t_k)A} v_{\varepsilon,k} \right\rangle \\
        &+ \lambda \left\langle y_0, e^{(t_{k+1}-t_k)A} v \right\rangle + \varepsilon \|y_0\|\| v_{\varepsilon,k} + \lambda v\|\\
        \leq &\dfrac{1}{2} \| \mathds{1}_{\omega}^{*} e^{(t_{k+1}-\tau_k)A} (v_{\varepsilon,k} + \lambda v) \|^2_{\omega} + \left\langle y_0, e^{(t_{k+1}-t_k)A} v_{\varepsilon,k} \right\rangle \\
        &+ \lambda \left\langle y_0, e^{(t_{k+1}-t_k)A} v \right\rangle + \varepsilon \|y_0\|\| v_{\varepsilon,k}\| + \varepsilon \left\lvert \lambda \right\rvert \|y_0\| \|v\|,
    \end{align*}
    which enables us to write
    \begin{equation}\label{eqq12}
        \begin{aligned}
        0\leq &\dfrac{1}{2} \left(\| \mathds{1}_{\omega}^{*} e^{(t_{k+1}-\tau_k)A} (v_{\varepsilon,k}+\lambda v)\|_{\omega} \mathds{1}_{\omega}^{*} e^{(t_{k+1}-\tau_k)A} v_{\varepsilon,k}\|_{\omega} \right)\nonumber\\
        &\times \left(\| \mathds{1}_{\omega}^{*} e^{(t_{k+1}-\tau_k)A} (v_{\varepsilon,k}+\lambda v)\|_{\omega} - \| \mathds{1}_{\omega}^{*} e^{(t_{k+1}-\tau_k)A} v_{\varepsilon,k}\|_{\omega}  \right)\nonumber\\
        &+ \lambda \left\langle y_0, e^{(t_{k+1}-t_k)A} v \right\rangle +  \varepsilon \left\lvert\lambda \right\rvert \|v\| \|y_0\|.
    \end{aligned}
    \end{equation}
    Then, for every $\lambda \neq 0$ the estimate \eqref{eqq12} leads us to
    \begin{align*}
        0\leq &\dfrac{1}{2} \left(\| \mathds{1}_{\omega}^{*} e^{(t_{k+1}-\tau_k)A} (v_{\varepsilon,k}+\lambda v)\|_{\omega} + \| \mathds{1}_{\omega}^{*} e^{(t_{k+1}-\tau_k)A} v_{\varepsilon,k}\|_{\omega} \right)\nonumber\\
        &\times \dfrac{1}{\left\lvert \lambda \right\rvert} \left(\| \mathds{1}_{\omega}^{*} e^{(t_{k+1}-\tau_k)A} (v_{\epsilon,k}+\lambda v)\|_{\omega} -\| \mathds{1}_{\omega}^{*} e^{(t_{k+1}-\tau_k)A} v_{\varepsilon,k}\|_{\omega} \right)\nonumber\\
        &+ \dfrac{\lambda}{\left\lvert\lambda \right\rvert} \left\langle y_0, e^{(t_{k+1}-t_k)A} v \right\rangle +  \varepsilon \|v\| \|y_0\|.
    \end{align*}
    Passing to the limit $\lambda \to 0^{+}$ and $\lambda \to 0^{-}$ in the above inequality, one can obtain 
    \begin{equation*}
        \left\lvert \left\langle \mathds{1}_{\omega}^{*} e^{(t_{k+1}-\tau_k)A} v_{\varepsilon,k} , \mathds{1}_{\omega}^{*} e^{(t_{k+1}-\tau_k)A} v\right\rangle+ \left\langle e^{(t_{k+1}-t_k)A} y_0,  v \right\rangle \right\rvert \leq \varepsilon \| v\| \|y_0\|.
    \end{equation*}
    Then from the above estimate and \eqref{normop}, it follows that
    \begin{equation}
        \left\lvert \left\langle e^{(t_{k+1}-t_k)A} y_0 + \mathds{1}_{\omega} e^{(t_{k+1}-\tau_k)A} u_{\varepsilon,k} , v \right\rangle \right\rvert \leq \varepsilon \| v \| \|y_0\|.
    \end{equation}
    On the other hand, the solution of the impulsive system \eqref{ACP1}corresponding to the initial data $y_0$ and the control function $u_{\epsilon,k}$ over the interval $(t_k, t_{k+1})$, can be expressed as follows
    \[y(t_k,t_{k+1},y_0,u_{\varepsilon,k})= e^{(t_{k+1}-t_k)A} y_0 + \mathds{1}_{\omega} e^{(t_{k+1}-\tau_k)A} u_{\varepsilon,k},\]
    this gives that
    \[\left\lvert \left\langle y(t_k,t_{k+1},y_0,u_{\varepsilon,k}) , v \right\rangle \right\rvert \leq \varepsilon \| v \| \|y_0\|, \;\;\forall v \in L^2(0,1),\]
    which implies 
    \begin{equation}\label{eqq8}
        \| y(t_k,t_{k+1}, y_0, u_{\varepsilon,k}) \| \leq \varepsilon \|y_0\|.
    \end{equation}
    This indicates that the impulsive system \eqref{ACP1} is null approximately controllable on the interval $(t_k, t_{k+1})$. Moreover, since the system \eqref{ACP1} is finite-time stabilizable, the control satisfies
    \[ \| \mathcal{L}_k(y(t_k)) \|^2_{\omega} \leq C_2 e^{-2k} \|y_0\|^2.\]
    Now, let us prove that for any $v \in L^2(0,1)$, one has
    \begin{equation}\label{eqq13}
        \| e^{(t_{k+1}-t_k)A} v \| \leq C e^{-k} \|y_0\|\| \mathds{1}_{\omega}^{*} e^{(t_{k+1}-\tau_k)A} v \|_{\omega} + \varepsilon\|y_0\| \| v\|.
    \end{equation}
    To this end, let us consider $v \in L^2(0,1)$ and multiply \(\eqref{ACP1}_{(1)}\) by $e^{(t_{k+1}-t)A} v$ and integrate over $(0,1)$ in order to derive
    \[\left\langle \partial_t y(t) , e^{(t_{k+1}-t)A} v \right\rangle - \left\langle A \,y(t) , e^{(t_{k+1}-t)A} v \right\rangle=0.\]
    The self-adjoint nature of the operator $A$ enables us to conclude that 
    \begin{equation}\label{eqq4}
        \partial_t \left( \left\langle y (t), e^{(t_{k+1}-t)A} v \right\rangle \right) =0.
    \end{equation}
    Integrating \eqref{eqq4} over the interval $\left[t_k , \tau_k\right)$ yields
    \begin{equation}\label{eqq5}
        \left\langle y(\tau_k^{-}) , e^{(t_{k+1}-\tau_k)A} v \right\rangle - \left\langle y(t_k) , e^{(t_{k+1}-t_k)A} v \right\rangle =0.
    \end{equation}
     Similarly, integrating over the interval  $\left(\tau_k, t_{k+1}\right)$ gives
    \begin{equation}\label{eqq6}
        \left\langle y(t_{k+1}) ,  v \right\rangle - \left\langle y(\tau_k) , e^{(t_{k+1}-\tau_k)A} v \right\rangle   =0.
    \end{equation}
    Combining the results from \eqref{eqq5} and \eqref{eqq6}, and using the fact that
    \[y(\tau_k) = y(\tau^{-}_{k}) + \mathds{1}_{\omega} \mathcal{L}_k (y(t_k)),\] 
    we can write
    \begin{equation}
        \left\langle y(t_{k+1}) , v \right\rangle = \left\langle y(t_k) , e^{(t_{k+1}-t_k)A} v \right\rangle + \left\langle \mathds{1}_{\omega} \mathcal{L}_k(y(t_k)) , e^{(t_{k+1}-\tau_k)A} v \right\rangle.
    \end{equation}
    Thus, for each $k\geq 0$,
    \begin{align*}
        \| e^{(t_{k+1}-t_k)A} v \|
        &= \sup_{\| y(t_k)\| \leq 1} \left\langle e^{(t_{k+1}-t_k)A} v , y(t_k) \right\rangle\nonumber\\
        &=\sup_{\| y(t_k)\| \leq 1} \left( \left\langle y(t_{k+1}) , v \right\rangle - \left\langle \mathcal{L}_k (y(t_k)) , e^{(t_{k+1}-\tau_k ) A} v \right\rangle_{\omega} \right)\nonumber\\
        &\leq \sup_{\| y(t_k)\| \leq 1} \left( \| y(t_{k+1})\| \| v\| + \| \mathcal{L}_k (y(t_k)) \|_{\omega} \| e^{(t_{k+1}-\tau_k)A} v\|_{\omega} \right).
    \end{align*}
    Given that $\| \mathcal{L}_k (y(t_k)) \|_{\omega} \leq C e^{-k} \| y_0\|$, and using this together with \eqref{eqq8}, we can obtain the desired estimate \eqref{eqq13}. Now, let us turn back to the fact that the functional $\mathcal{J}_{\varepsilon,k}$ admits a unique minimum $v_{\epsilon,k}$, this enables us to derive
    \[\mathcal{J}^{'}_{\varepsilon,k} (v_{\varepsilon,k}) \cdot w =0, \;\;\forall w \in L^2(0,1),\]
    which implies that
    \[\left\langle \mathds{1}^{*}_{\omega} e^{(t_{k+1}-\tau_k)A} v_{\varepsilon,k}, \mathds{1}^{*}_{\omega} e^{(t_{k+1}-\tau_k)A} w \right\rangle_{\omega} + \left\langle y_0 , e^{(t_{k+1}-t_k)A} w \right\rangle_{\omega} + \varepsilon \|y_0\|\dfrac{\left\langle v_{\varepsilon,k} , w \right\rangle}{\| v_{\varepsilon,k} \|} =0,\]
    take $w=v_{\varepsilon,k}$ in the above quantity and using the estimate \eqref{eqq13} to get
    \begin{align}\label{eqq9}
       & \| \mathds{1}^{*}_{\omega} e^{(t_{k+1}-\tau_k)A} v_{\varepsilon,k} \|_{\omega}^2\nonumber\\
        &= - \left\langle y_0 , e^{(t_{k+1}-t_k)A} v_{\varepsilon,k} \right\rangle_{\omega} - \varepsilon \|y_0\| \| v_{\varepsilon,k} \|\nonumber\\
        &\leq \| y_0 \| \| e^{(t_{k+1}-t_k)A} v_{\varepsilon,k} \| - \varepsilon \|y_0\| \| v_{\varepsilon,k} \|\nonumber\\
        &\leq \left( \varepsilon \|y_0\| \| v_{\varepsilon,k} \| + C e^{-k} \| \mathds{1}^{*}_{\omega} e^{(t_{k+1}-t_k)A} v_{\varepsilon,k} \|_{\omega} \| y_0 \| \right) \| y_0\| - \varepsilon \|y_0\|\| v_{\varepsilon,k} \|\nonumber\\
        &=C e^{-k} \| \mathds{1}^{*}_{\omega} e^{(t_{k+1}-t_k)A} v_{\varepsilon,k} \|_{\omega} \| y_0 \|^2.
    \end{align}
    That allows us to give
    \[\| \mathds{1}^{*}_{\omega} e^{(t_{k+1}-\tau_k)A} v_{\varepsilon,k} \|_{\omega} \leq C e^{-k} \| y_0 \|^2.\]
    By summing over $k \geq 0$, one has
    \begin{align*}
        \sum_{k \geq 0} \| \mathds{1}^{*}_{\omega} e^{(t_{k+1}-\tau_k)A} v_{\varepsilon,k} \|_{\omega} 
        \leq C \sum_{k \geq 0} e^{-k} \| y_0 \|^2
        =\dfrac{C}{1-e^{-1}} \| y_0\|^2.
    \end{align*}
   Then, we infer that there exists a positive constant $\mathcal{K}$ such that
   \begin{equation*}
       \| (\mathds{1}^{*}_{\omega} e^{(t_{k+1}-\tau_k)A} v_{\varepsilon,k})_{k \geq 0} \|_{\ell^2(L^2(\omega))} \leq \mathcal{K} \| y_0\|^2, 
   \end{equation*}
   from \eqref{normop}, the above inequality becomes
   \begin{equation}\label{eqq10}
       \| ( u_{\varepsilon,k})_{k \geq 0} \|_{\ell^2(L^2(\omega))} \leq \mathcal{K} \| y_0\|^2.
   \end{equation}
   This allows us to deduce that there is $(u_k)_{k \geq 0} \in \ell^2(L^2(\omega))$ such that 
   \[u_{\varepsilon,k} \to u_k \;\;\text{as}\;\varepsilon \to 0 \;\;\text{weakly in}\;\;\ell^2(L^2(\omega)).\]
   By the linearity of $(y_0,v) \mapsto y(T, y_0, v)$, one has
   \begin{equation*}
       y(T,y_0,(u_{\varepsilon,k})_{k \geq 0}) \to y(T,y_0,(u_k)_{k \geq 0}) \;\;\text{as}\;\;\varepsilon\to 0\;\;\text{weakly in}\;\;L^2(0,1).
   \end{equation*}
   Moreover, one has
   \begin{align*}
       \| y(T,y_0,(u_{\varepsilon,k})_{k \geq 0})\|
       &\leq e^{\frac{T}{b^{k+1}}\delta} \| y(t_{k+1},y_0,(u_{\varepsilon,k})_{k \geq 0})\| \\
       &\leq \varepsilon \, e^{\frac{T}{b^{k+1}}\delta}\|y_0\| \to 0 \qquad\text{as}\;\;\varepsilon\to 0.
   \end{align*}
   This implies that 
   \[y(T,y_0,(u_{\varepsilon,k})_{k \geq 0})=0.\]
    Hence, it follows that \( (u_{\varepsilon,k})_{k \geq 0} \) constitutes a solution to problem \( (\mathcal{P}) \). Now, we proceed to establish the uniqueness of the solution. To do so, we assume \( u_1 \) to be an arbitrary null control for \eqref{ACP1}. Consequently,
    \begin{equation*}
       \| (u_{\varepsilon,k})_{k \geq 0} \|_{\ell^2(L^2(\omega))}\leq \| u_1\|_{\ell^2(L^2(\omega))}.
    \end{equation*}
    Since $\| (u_k)_{k \geq 0}\|_{\ell^2(L^2(\omega))} \leq \lim \inf \|(u_{\epsilon,k})_{k \geq 0}\|_{\ell^2(L^2(\omega))}$, then
   \begin{equation*}
       \| (u_k)_{k \geq 0}\|_{\ell^2(L^2(\omega))} \leq \|u_1\|_{\ell^2(L^2(\omega))}.
   \end{equation*}
   This proves that $(u_k)_{k \geq 0}$ is norm optimal control of $(\mathcal{P})$.
\end{enumerate}
\end{proof}

\section{Conclusions and possible extensions}
In this study, we generalize the findings of \citep{MMOS} regarding impulsive null approximate controllability for singular and degenerate parabolic equations, establishing impulse null controllability through a sequence of pulses. To achieve this, we provide an explicit estimate of the exponential decay of the solution via impulse controls. Additionally, we address the norm-optimal impulsive control problem.

It is important to note that this work necessitates considering a sequence of carefully chosen pulses over the time horizon to establish impulse null controllability. This raises the intriguing question of whether null controllability can be achieved with a single pulse. Such a scenario represents a significant generalization of many works on the controllability of parabolic equations, as the impulsive control is a very weak control that acts only at a single instant in time and within an arbitrarily small region of the physical domain.

Another promising avenue for extending this work involves investigating the numerical controllability aspect, similar to the approach taken in \cite{CGMZ'221} for the one-dimensional heat equation.

\section*{Declarations}

\subsection*{Ethical Approval}
This study did not involve human participants or animals. Therefore, ethical approval and informed consent were not required. This declaration is not applicable.

\subsection*{Funding}
No funding was received to support this research. This declaration is not applicable.

\subsection*{Clinical Trial Number}
This study did not involve a clinical trial. This declaration is not applicable.

\subsection*{Disclosure Statement}
The authors declare that they have no conflicts of interest regarding the publication of this article.

\bibliography{sn-bibliography}

@article{ABD2006,
  author		= " Ammar-Khodja, F. and  Benabdallah, A. and  Dupaix, C.",
  title			= "Null-controllability of some reaction-diffusion systems with one control force",
  journal		= "J. Math. Anal. Appl.",
  volume		= "320",
  pages			= "928--943",
  year			= "2006"
}

@article{ABDG2009,
  author		= "Ammar Khodja, F.  and  Benabdallah, A. and  Dupaix, C. and  Gonzalez-Burgos, M.",
  title			= "A generalization of the Kalman rank condition for time-dependent coupled linear parabolic systems",
  journal		= "Differ. Equ. Appl.",
  volume		= "1",
  pages			= "427-457",
  year			= "2009"
}

@article{AJMW2024,
  author		= " Ait Ben Hassi, E. M. and Jakhoukh, M. and Maniar, L. and Zouhair, W.",
  title			= " Internal null controllability for the one-dimensional heat equation with dynamic boundary conditions",
  journal		= "IMA J. Math. Control. Inf.",
  volume		= "41",
  pages			= "403–424",
  year			= "2024"
}

@article{CGKM2023,
  author		= " Chorfi, S. E. and El. Guermai, G. and Khoutaibi, A. and Maniar, L.",
  title			= "Boundary null controllability for the heat equation with dynamic boundary conditions",
  journal		= "Evol. Equ. Control Theory",
  volume		= "12",
  pages			= "542--566",
  year			= "2023"
}

@article{BCG,
author = "Alabau-Boussouira, F. and Cannarsa, P. and Fragnelli, G.",
year = "2006",
pages = "161-204",
title = "Carleman estimates for degenerate parabolic operators with applications to null controllability",
volume = "6",
journal = "J. Evol. Equ.",
}

@article{BAJS,
  author		= "Allal, B. and Salhi, J.",
  title			= "Pointwise Controllability for Degenerate Parabolic Equations by the Moment Method",
  journal		= "J. Dyn. Control Syst.",
  volume		= "26",
  pages			= "349--362",
  year			= "2020"
}

@article{CMV,
author = "Cannarsa, P. and Martinez, P. and , Vancostenoble, J.",
year = "2005",
title = "Null controllability of degenerate heat equations",
volume = "10",
journal = "Adv. Differential Equations",
 pages = "153--190",
}

@article{AFS,
author = "Allal, B. and Fragnelli, G. and Salhi, J.",
year = "2022",
title = "Controllability for degenerate/singular parabolic systems involving memory terms",
journal = "Discrete Contin. Dyn. Syst. - S",
 pages = "1",
}

@article{FS,
author = "Fotouhi, M. and Salimi, L.",
title = "Null controllability of degenerate/singular parabolic equations",
volume = "18",
journal = "J. Dyn. Control Syst.",
year = "2012",
}

@article{V,
author = "Vancostenoble, J.",
title = "Improved Hardy-Poincaré inequalities and sharp Carleman estimates for degenerate/singular parabolic problems",
journal		= "Discrete Contin. Dyn. Syst. - S",
  volume		= "4",
  pages			= "761-790",
  year			= "2010"
}

@article{pkm,
  author		= " Phung, K. D.",
  title			= " Carleman commutator approach in logarithmic convexity for parabolic equations",
  journal		= "Math. Control Rel. Fields",
  volume		= "8",
  pages			= "899--933",
  year			= "2018"
}

@article{RBKDP2022,
  author		= " Buffe, R. and Phung, K. D.",
  title			= "Observation estimate for the heat equations with Neumann boundary condition via logarithmic convexity",
  journal		= "J. Evol. Equ.",
  volume		= "22",
  pages			= "86",
  year			= "2022"
}

@article{Coron2017,
  author		= " Coron , J. M. and Nguyen, H. M.",
  title			= "Null controllability and finite time stabilization for the heat equations with variable coefficients in space in one dimension via backstepping approach",
  journal		= "Arch Ration Mech Anal .",
  volume		= "225",
  pages			= "993--1023",
  year			= "2017"
}

@article{pkm1,
  author		= "Phung, K. D., Wang G. and Xu. Y",
  title			= "Impulse output rapid stabilization for heat equations",
  journal		= "J. Differential Equations",
  volume		= "263",
  pages			= "5012–5041",
  year			= "2017"
}

@article{MMOS,
    author = "Maarouf, H. Maniar, L. and Ouelddris, I. and Salhi, J.",
    title = "Impulse controllability for degenerate singular parabolic equations via logarithmic convexity method",
    journal = "IMA J. Math. Control Inf.",
    volume		= "40",
    pages = "1471-6887",
    year = "2023"
}

@article{FO,
author = "Frigon, M. and Oregan, D.",
year = "1995",
title = "Existence results for first-order impulsive differential equations",
volume = "193",
pages="96-113",
journal=". J. Math. Anal.Appl."
}

@article{SLAPWZ2,
  author		= " Lalvay, S. and Padilla-Segarra, A. and  Zouhair, W.",
  title			= "On the existence and uniqueness of solutions for non-autonomous semi-linear systems with non-instantaneous impulses, delay, and non-local conditions",
  journal		= "Miskolc Math. Notes",
  volume		= "23",
  pages			= "295--310",
  year			= "2022"
}

@article{Gal2015,
  author		= " Gal, C. G.",
  title			= "The role of surface diffusion in dynamic boundary conditions: Where do we stand?.",
  journal		= "Milan J Math.",
  volume		= "83",
  pages			= "237-278",
  year			= "2015"
}


\end{document}